\documentclass{amsart}
\usepackage{verbatim}
\usepackage{rotating} 
\usepackage{amssymb}
\usepackage{mathrsfs,amsfonts,amsmath,amssymb,epsfig,amscd,xy,amsthm,pb-diagram} 

\newtheorem{theorem}{Theorem}[section]
\newtheorem{proposition}[theorem]{Proposition}

\newtheorem{lemma}[theorem]{Lemma}

\newtheorem{cor}[theorem]{Corollary}
\newtheorem{question}[theorem]{Question}
\newtheorem{definition}[theorem]{Definition}

\theoremstyle{plain}

\theoremstyle{remark}

\newtheorem{remark}[theorem]{Remark}

\newcommand{\C}{{\mathbb C}}

\newcommand{\Q}{{\mathbb Q}}

\newcommand{\R}{{\mathbb R}}
\newcommand{\Z}{{\mathbb Z}}

\newcommand{\A}{{\mathbb A}}

\newcommand{\bP}{{\mathbb P}}

\author{Khoa Nguyen}
\address{
Khoa Nguyen \\
Department of Mathematics\\
University of California\\
Berkeley, CA 94720 
}

\email{khoanguyen2511@gmail.com}
\urladdr{www.math.berkeley.edu/\~{}khoa}


\begin{document}
	\title[Independence of Local Conjugacies and Related Questions]{Algebraic Independence of Local Conjugacies and Related Questions in Polynomial Dynamics}
	\date{10/9/2013 (version 1).}
	\begin{abstract}		 
	  Let $K$ be an algebraically closed field of characteristic 0 and $f\in K[t]$
	  a polynomial of degree $d\geq 2$. There exists a local conjugacy $\psi_f(t)\in 
	  tK[[1/t]]$
	  such that $\psi_f(t^d)=f(\psi_f(t))$. It has been known that $\psi_f$ is transcendental
	  over $K(t)$ if $f$ is not conjugate to
	  $t^d$ or a constant multiple of the Chebyshev polynomial. In this paper, we study the
	  algebraic independence of $\psi_{f_1}$,\ldots,$\psi_{f_n}$ 
	  using a recent result of Medvedev-Scanlon. Related questions in transcendental number
	  theory and canonical heights in arithmetic dynamics are also discussed.
	\end{abstract}
	\maketitle
	\section{Introduction and Main Results} \label{sec:intro}
	Throughout this paper, let $K$ be an algebraically closed field of
	characteristic 0. Let $f\in K[t]$ be a polynomial of degree $d\geq 2$. 
	There exists a Laurent series:
	$$\psi_f(t)=a_{-1}t+a_0+\frac{a_1}{t}+\frac{a_2}{t^2}+\ldots\in tK[[1/t]],\ a_{-1}\neq 
	0$$
	such that 
	\begin{equation}\label{eq:functional}
	\psi_f(t^d)=f(\psi_f).
	\end{equation}
	Moreover, another series $\tilde{\psi_f}$ satisfies the above conditions
	if and only if there exists a $(d-1)$th-root of unity $\zeta$
	such that $\tilde{\psi_f}(t)=\psi_f(\zeta t)$. To prove these claims, 
	we may use the functional equation
	(\ref{eq:functional}) to either write
	down and solve the system of equations involving the coefficients of $\psi_f$,
	or notice that these coefficients belong to the algebraic closure $F$
	of the field generated by the coefficients of $f$, and embed $F$ into $\C$.
	Over $\C$, we have classical results in the complex dynamics
	of polynomials (see, for example, \cite[p.~98]{MilnorBook}) which
	motivate the results of this paper. In fact, the inverse $\psi_f^{-1}\in tK[[1/t]]$, which is
	well-defined up to multiplication by a $(d-1)$th-root of unity,
	is called the B\"ottcher coodinate of $f$.
		
	
	The Chebyshev polynomial
	$C_d(t)$ is defined to be the unique polynomial
	of degree $d$ such that 
	$C_d(t+\displaystyle\frac{1}{t})=t^d+\displaystyle\frac{1}{t^d}$.
	When $f(t)=t^d$ or $f(t)=C_d(t)$, we may choose $\psi_f(t)=t$, or $\psi_f(t)=t+\displaystyle\frac{1}{t}$ respectively.
	When $f(t)=-C_d(t)$, we pick any $\mu\in K$ such that $\mu^{d-1}=-1$. Then
	we may choose
	 $\psi_f(t)=\displaystyle\frac{t}{\mu}+\displaystyle\frac{\mu}{t}$.
	Therefore, if $f$ is conjugate to $t^d$ or $\pm C_d(t)$, we have
	that $\psi_f\in K(t)$. 
	
	Following Medvedev-Scanlon \cite[Fact 2.25]{MedSca},
	the polynomial $f$ is said to be disintegrated if
	$f$ is not conjugate to $t^d$
	or $\pm C_d(t)$. 
	In \cite{BB93}, Becker and Bergweiler
	prove that $\psi_f$ is transcendental over $K(t)$ if
	$f$ is disintegrated. Shortly after that, they
	prove the stronger conclusion that $\psi_f$ is
	hypertranscendental \cite{BB95}. On the other hand,
	it is a well-known result in polynomial dynamics that 
	for every disintegrated polynomial $f$ and $g$, there is a choice  
	of $\psi_f$ and $\psi_g$ such that $\psi_f=\psi_g$ if and only if
	$f$ and $g$ have a common iterate. 
%
%
%
	We include a proof of this due to the lack of
	an immediate reference. Let $d$ and $e$ denote the degrees of $f$ and $g$ 
	respectively. By using (\ref{eq:functional}), it is
	easy to prove that if $f$ and $g$ have a common iterate
	then we can make the choice $\psi_f=\psi_g$. 
	Now assume that $\psi_f(t)=\psi_g(t)$, we have 
	$\psi_g(t^d)=\psi_f(t^d)=f\circ \psi_f(t)$.
	Therefore:
	$$\psi_g(t^{de})=g\circ\psi_g(t^d)=g\circ f\circ\psi_f(t).$$
	Interchanging $f$ and $g$, we have:
	$$\psi_f(t^{de})=f\circ g\circ\psi_g(t)$$
	Therefore $g\circ f=f\circ g$. A classical theorem of Ritt \cite[p.~399]{Ritt23}
	concludes that $f$ and $g$ have a common iterate. 
	
  
  We will provide a characterization of disintegrated polynomials
  $f$ and $g$ such that $\psi_f(t)$ and $\psi_g(t)$ 
  are algebraically dependent over
  $K(t)$. Therefore in a certain sense, the results in this paper might be 
  regarded 
  as a
  common extension to results stated in the last paragraph. First, we need
  the following definition from dynamics:
  \begin{definition}
  Let $p(t)$ and $q(t)$ be non-linear polynomials in $K[t]$. We say that 
  $p(t)$ is semi-conjugate to $q(t)$ if there exists a non-constant 
  polynomial $\pi(t)\in K[t]$
  such that $p\circ \pi=\pi\circ q$.
  \end{definition}
  
  In fact, it is more natural 
  to consider algebraic dependence of $\psi_f(\alpha t^{m})$,
  and $\psi_g(\beta t^{n})$ for certain roots of unity $\alpha,\beta$ and positive
  integers $m,n$. For $d\geq 2$, let $\mu_{(d)}$ denote
  the set of all roots of unity $\zeta$ whose order is relatively prime to $d$.
  Write:
  $$M_d=\left\{\zeta t^m:\ \zeta\in \mu_{(d)},\ m\in\Z,\ m>0\right\}.$$
	It is not difficult to prove that $M_d$ is exactly the set of all non-constant 
	polynomials commuting with an iterate of $t^d$. Our main results can now
	be formulated independently from the choices
	of roots of unity in the definition of $\psi_f$:
	
	\begin{theorem}\label{thm:main1}
	Let $f,g\in K[t]$ be disintegrated polynomials of degrees $d,e\geq 2$ 
	respectively. 
	There exist $u(t)\in M_d$, and $v(t)\in M_e$ such that $\psi_f(u(t))$ and 
	$\psi_g(v(t))$
	are algebraically dependent over $K(t)$
	if and only if an iterate of $f$ and 
	an iterate of $g$ are semi-conjugate to a common 
	polynomial.
	\end{theorem}
	 
	We give a similar result in the case of more than 2 
	polynomials \textit{under the condition that all the degrees have a common power}.
	Our first step in proving Theorem \ref{thm:main1}
	is to show that $\deg(f)$ and $\deg(g)$ have a common power, then apply the
	following:
	\begin{theorem}\label{thm:main2}
	Let $n\geq 2$, and let $f_1,\ldots,f_n\in K[t]$ be polynomials
	whose degrees have a common power $d\geq 2$. For $1\leq i\leq n$,
	write $\psi_i=\psi_{f_i}$. There exist 
	$u_1(t),\ldots,u_n(t)\in M_d$ such that $\psi_1(u_1(t)),\ldots,\psi_n(u_n(t))$
	are algebraic dependent over $K(t)$ if and only if
	there exist $1\leq i\neq j\leq n$ such that an iterate of $f_i$
	and an iterate of $f_j$ are semi-conjugate to a common polynomial.
	\end{theorem}

	The organization of this paper is as follows. We first exhibit an algebraic
	dependence relation between $\psi_f(u(t))$ and $\psi_g(v(t))$
	for some $u,v\in M_d$ when an iterate of $f$ and an iterate of $g$
	are semi-conjugate to a common polynomial. After that we prove Theorem
	\ref{thm:main1}, and Theorem \ref{thm:main2}
	using the main result of Medvedev-Scanlon \cite{MedSca}.
	Then we provide an application of our results
	to the simplest one-parameter family of rational maps, namely the 
	family $\{t^2+c:\ c\in K\}$. Finally we conclude the paper by proposing some 
	related questions
	in transcendental number theory and arithmetic dynamics.
	
	{\bf Acknowledgments.}  The author would like to thank Tom Tucker and Tom Scanlon
	for many useful suggestions. We are grateful to Patrick Ingram and Joe Silverman for
	helpful conversations involving Section \ref{sec:questions}.

	 \section{Constructing Algebraic Dependence Relations}
	 Let $v_{\infty}$ denote the usual discrete
	 valuation on the field $K((1/t))$ of formal Laurent series in 
	 $1/t$. We have the following useful result:
	 \begin{lemma}\label{lem:D}
	 Let $f\in K[t]$ be a  polynomial of degree $d\geq 2$. Let $D\geq 1$. Let 
	 $L(t)\in K((1/t))$ such that $v_\infty(L)=D$
	 and $L(t^d)=f(L(t))$. Then there exists $\psi\in K((1/t))$
	 such that $v_\infty(\psi)=1$, 
	 $\psi(t^d)=f(\psi(t))$, and $L(t)=\psi(t^D)$.
	 \end{lemma}
	 
	 \begin{proof}
	 There is nothing to prove when $D=1$. We may assume that $D\geq 2$.
	 We prove the lemma by using contradiction: let $M$ be the largest integer such 
	 that $D$ does not divide $M$ and $c_M\neq 0$. Write:
	 \begin{align*}
	 f(t)&=b_dt^d+\ldots+b_0.\\
	 f(L(t))&=
b_d(c_Dt^D+\ldots+c_Mt^M+\ldots)^d
+b_{d-1}(c_Dt^D+\ldots+c_Mt^M+\ldots)^{d-1}+\ldots
	 \end{align*}
	 
	 By the choice of $M$,
	 the coefficient of $t^{D(d-1)+M}$ in $f(L(t))$
	 is $db_dc_D^{d-1}c_M$. On the other hand, the coefficient of
	 $t^{D(d-1)+M}$ in $L(t^d)$ is either 0 
	 or $c_N$ if 
	 $N= \displaystyle \frac{D(d-1)+M}{d}$ is an integer.
	 In this latter case, we have $c_N=0$
	 since $N$ is not a multiple of $D$, $N>M$, and
	 the maximality of $M$.
	 Therefore, comparing the coefficient of $t^{D(d-1)+M}$, we have:
	 $$db_dc_D^{d-1}c_M=0.$$
	 Hence $c_M=0$, contradiction.
	 
%

   Therefore $L(t)=\psi(t^D)$
	 for some $\psi\in K((1/t))$, $v_\infty(\psi)=1$.	We have:
	 $$f(\psi(t^D))=f(L(t))=L(t^d)=\psi(t^{Dd}).$$
	 Hence $f(\psi(t))=\psi(t^d)$. 
	 \end{proof}
	
	Lemma \ref{lem:D} is used to prove the following:
	\begin{proposition}\label{prop:algdep}
	Let $f,g\in K[t]$ be polynomials of degree $d\geq 2$ that are semi-conjugate
	to a common polynomial $h$. Let $\pi,\rho\in K[t]$ be non-constant
	such that $f\circ \pi=\pi\circ h$, and $g\circ \rho=\rho\circ h$.
	 Then there exists a choice of $\psi_f$ and $\psi_g$
	such that $\psi_f(t^{\deg(\pi)})$ and $\psi_g(t^{\deg(\rho)})$
	are algebraically dependent over $K$.
	\end{proposition} 
	\begin{proof}
	Fix a choice of $\psi_h(t)$, let $L(t)=\pi(\psi_h(t))$. We have that 
	$v_\infty(L)=\deg(\pi)$, and 
	$L(t^d)=f(L(t))$ thanks to the functional equation $f\circ\pi=\pi\circ h$.
	By Lemma \ref{lem:D}, there is a choice of $\psi_f$ such that
	$\psi_f(t^{\deg(\pi)})=L(t)=\pi(\psi_h(t)).$ Similarly, there is a choice
	of $\psi_g$ such that $\psi_g(t^{\deg(\rho)})=\rho(\psi_h(t))$.
	Therefore $\psi_f(t^{\deg(\pi)})$ and  $\psi_g(t^{\deg(\rho)})$
	are algebraically dependent over $K$ since both are in the field
	$K(\psi_h)$.
	\end{proof} 
	 
	\begin{cor}\label{cor:algdep}
		Let $f,g\in K[t]$ such that
		for some positive integers $n,m$, there exist
		non-constant $\pi,\rho\in K[t]$ and a polynomial $h\in K[t]$
		of degree $\delta\geq 2$ satisfying
		$f^n\circ\pi=\pi\circ h$,
		and $g^m\circ \rho=\rho\circ h$.
		Fix a choice of
		$\psi_f$ and $\psi_g$. Then there exist $(\delta-1)$th roots of unity
		$\alpha$, $\beta$ such that
		$\psi_f((\alpha t)^{\deg(\pi)})$ and 
		$\psi_g((\beta t)^{\deg{\rho}})$
		are algebraically dependent over $K$.
	\end{cor} 
	 \begin{proof}
	 	By Proposition \ref{prop:algdep}, there is a choice of $\psi_{f^m}$
	 	and $\psi_{g^n}$ such that $\psi_{f^m}(t^{\deg(\pi)})$
	 	and $\psi_{g^n}(t^{\deg(\rho)})$ are algebraically dependent over $K$.
	  Since $\psi_f$ and $\psi_g$ are other choices for $\psi_{f^m}$
	  and $\psi_{g^n}$ respectively, there exist $(\delta-1)$th roots of unity
	  $\alpha$ and $\beta$ such that $\psi_{f^m}(t)=\psi_f(\alpha t)$
	  and $\psi_{g^n}(t)=\psi_g(\beta t)$. This finishes the proof.
	 \end{proof}
	
	\begin{remark}\label{rem:overK}
		Note that Proposition \ref{prop:algdep} and Corollary \ref{cor:algdep}
		provide algebraic dependence over $K$, not just $K(t)$. 
	\end{remark}
	
	\section{Proof of Theorem \ref{thm:main2}}
	If an iterate of $f_i$ and an iterate of $f_j$ are semi-conjugate to a common
	polynomial then we apply Corollary \ref{cor:algdep} (see
	Remark \ref{rem:overK}). It remains to
	prove the ``only if'' part.
	Replace $d$ by an appropriate power such that each $u_i(t)$ commutes
	with $t^d$. Replacing each
	 $f_i$ by
	an iterate, we may assume that all the polynomials $f_1,\ldots,f_n$ have
	a common degree $d$. By relabeling and removing the $f_i$, and lowering $n$
	if necessary, we may assume
	that $\psi_{1}(u_1(t)),\ldots,\psi_{n-1}(u_{n-1}(t))$
	are algebraically independent over $K(t)$ and 
	$\psi_n(u_n(t))$ is algebraic of degree $D\geq 1$ over 
	$K(t,\psi_1(u_1(t)),\ldots,\psi_{n-1}(u_{n-1}(t)))$.
%
%
  We have the following irreducible polynomial in $K[X_0,X_1,\ldots,X_n]$:
$$P(X_0,X_1,\ldots,X_n)=P_D(X_0,\ldots,X_{n-1})X_n^D+\ldots+P_0(X_0,\ldots,X_{n-1})$$
		such that $P(t,\psi_1(u_1(t)),\ldots,\psi_n(u_n(t)))=0$ and $\gcd(P_D,\ldots,P_0)=1$. In fact, $P$ is the minimal polynomial of
		$\psi_n(u_n)$ over
		$$K[X_0,\ldots,X_{n-1}]\cong K[t,\psi_1(u_1),\ldots,\psi_{n-1}(u_{n-1})].$$
	
	Replacing $t$ by $t^d$ and using the identities $\psi_i(t^d)=f_i(\psi_i)$,
	and $u_i(t^d)=(u_i(t))^d$ for each $i$, we have:
	$$P(t^d,f_1(\psi_1(u_1)),\ldots,f_n(\psi_n(u_n)))=0.$$	
	Therefore, we have that $P(X_0,\ldots,X_n)$ is a factor of 
	$P(X_0^d,f_1(X_1),\ldots,f_n(X_n))$. Consider the self-map
	$\varphi$ of $\A^{n+1}$ given by:
	$$\varphi(x_0,\ldots,x_n)=(x_0^n,f_1(x_1),\ldots,f_n(x_n)).$$
	Let $V$ be the hypersurface defined by $P(X_0,\ldots,X_n)=0$, then we have
	$V\subseteq \varphi^{-1}(V)$. In other words, $V$ is an invariant hypersurface
	in the dynamical system $(\A^n,\varphi)$. For each $1\leq i\neq j\leq n$,
	let $\pi_{i,j}$ denote the projection from $\A^n$
	onto the $(i,j)$-factor $\A^2$.
	By the main theorem of Medvedev-Scanlon
	\cite[Theorem 2.30]{MedSca}, there exists $1\leq i\neq j\leq n$ such that
	  $\pi_{i,j}(V)$ is an invariant curve
	  of $\A^2$ under the self-map $(f_i,f_j)$, and: 
	  $$V=\pi_{i,j}^{-1}(\pi_{i,j}(V))\cong\pi_{i,j}(V)\times \A^{n-1}$$
%
	 	
	Since $X_n$ cannot be the only variable that appears in $P$, there are at least
	2 variables appearing in $P$. Hence the invariant curve $\pi_{i,j}(V)$
	cannot map to a single point in each factor $\A^1$ of $\A^2$. By 
	\cite[Proposition 2.34]{MedSca}, the polynomials $f_i$ and $f_j$
	are semi-conjugate to a common polynomial. This finishes the proof.
	
	\begin{remark}\label{rem:simplecase}
		It follows from the proof that if  $f_1,\ldots,f_n$ have the same degree $d$,     $u_1,\ldots,u_n$ commute with $t^d$, and $\psi_1(u_1),\ldots,\psi_n(u_n)$
		are algebraically dependent over $K(t)$ then there exist $1\leq i\neq j\leq n$
		such that $f_i$ and $f_j$ are semi-conjugate to a common polynomial.
	\end{remark}
	
	\begin{remark}\label{rem:moregeneral}
		The arguments here can be used in a more general setting, as follows. Let
		$f_0(t),\ldots,f_n(t)$ be polynomials, let $\psi_1(t),\ldots,\psi_n(t)$
		be formal series such that $\psi_i(f_0(t))=f_i(\psi_i(t))$ for
		every $1\leq i\leq n$. Then algebraic dependence of $\psi_1,\ldots,\psi_n$
		over $K(t)$ would give rise to an invariant hypersurface
		of the dynamical system $(\A^{n+1},\varphi)$ where
		$\varphi=(f_0,\ldots,f_n)$. Then the main result of
		Medvedev-Scanlon can be applied. 
	\end{remark}
	
	\section{Proof of Theorem \ref{thm:main1}}
	The ``if'' part follows immediately from Corollary \ref{cor:algdep}. Before proving
	the ``only if'' part, we prove the following theorem which will also be needed 
	later. We let $\tilde{f}$ denote a polynomial of smallest degree
	such that $f$ is an iterate of $\tilde{f}$.
	Write $\delta=\deg(\tilde{f})$. It is
	well-known that
	every polynomial commuting with an iterate of $f$ must have the form
	$\tilde{f}^n\circ l$ where $l$ is a linear polynomial commuting with an iterate of 
	$f$ (see, for example, \cite[Proposition 2.3]{splitv2} and 
	\cite[p.~12]{splitv2}). Therefore, the
	degree of every such
	polynomial is a power of $\delta$.
	
	\begin{theorem}\label{thm:samepol}
		Assume $n\geq 2$, and $\alpha_1(t),\ldots,\alpha_n(t)\in M_d(t)$ such 
		that the functions
		$\psi_f(\alpha_1(t)),\ldots,\psi_f(\alpha_n(t))$ are algebraically
		dependent over $K(t)$. Then there exist $1\leq i\neq j\leq n$ such that
		$\displaystyle\frac{\deg(\alpha_i)}{\deg(\alpha_j)}$
		has the form $\delta^m$ for some positive integer $m$.
%
	\end{theorem}
	\begin{proof}
		Replacing $f$ by an iterate, and hence $d$ by its power such that 
		 $\alpha_i(t)$ commutes with $t^d$ for every $1\leq i\leq n$.
		As in the proof of Theorem \ref{thm:main2} in the previous section, we may assume
		that $\psi_f(\alpha_1),\ldots,\psi_f(\alpha_{n-1})$
		form a  transcendence basis of
		$K(t,\psi_f(\alpha_1),\ldots,\psi_f(\alpha_n))$
		over $K(t)$, and 
		let $P(X_0,\ldots,X_{n})$ be the minimal polynomial of $\psi_f(\beta)$
		over $K[X_0,\ldots,X_{n-1}]\cong K[t,\psi_f(\alpha_1),\ldots,\psi_f(\alpha_{n-1})]$. 
		As before,
		we have that the hypersurface $V$ in $\A^{n+1}$ defined by
		$P(X_0,\ldots,X_n)=0$ is invariant under the self-map
		$$\varphi(x_0,x_1,\ldots,x_n)\colon=(x_0^d,f(x_1),\ldots,f(x_n)).$$
		
		By Theorem 2.30 and Theorem 6.24 of Medvedev-Scanlon \cite{MedSca},
		there exists $1\leq i\neq j\leq n$ such that 
		$P(X_0,X_1,\ldots,X_n)$ is $X_i-Q(X_j)$ where
		$Q$ is a polynomial commuting with $f$. Therefore we must have
		$\psi_f(\alpha_i)=Q(\psi_f(\alpha_j))$. Applying the valuations $v_\infty$, we have
		$\deg(\alpha_i)=\deg(Q)\deg(\alpha_j)$. Therefore
		$$\frac{\deg(\alpha_i)}{\deg(\alpha_j)}=\deg(Q)$$
		is a power of $\delta$.		
	\end{proof}
	
	\begin{remark}
    In a sense, Theorem \ref{thm:samepol}
		provides a complement to Theorem \ref{thm:main2}
		by treating the case $f_1=\ldots=f_n$.	
	\end{remark}
	
	We can now prove Theorem \ref{thm:main1} as follows. By Theorem \ref{thm:main2},
	it suffices to show that $d=\deg(f)$ and $e=\deg(g)$ have a common
	power. Replacing $f$ and $g$ by appropriate iterates so that $d$ and $e$ are 
	replaced
	by the corresponding powers, we may assume that $u(t)$ and $v(t)$, 
	respectively, 
	commute
	with $t^d$ and $t^e$. Since $\psi_f(u(t))$ and $\psi_g(v(t))$ are algebraically
	dependent over $K(t)$, so are $\psi_f(u(t^e))$
	and $\psi_g(v(t^e))$. Since $\psi_g(v(t^e))=g(\psi_g(v(t)))$
	and $\psi_f(u(t))$
	are algebraically dependent over $K(t)$, 
	so are $\psi_f(u(t^e))$ and $\psi_f(u(t))$.
	By Theorem \ref{thm:samepol}, $e=\displaystyle\frac{\deg(u(t^e))}{\deg(u(t))}$
	has a common power with $d$. This finishes the proof of Theorem \ref{thm:main1}.
	
	\section{An Example}
	Our main results are useful since they can be combined to
	the classification of semi-conjugate polynomials given in the appendix
	of Inou's paper
	\cite[Appendix A]{Inou}
	using results of Ritt \cite{Ritt22} and Engstrom \cite{Engstrom}.
	For example, we consider the simplest one-parameter family of
	rational maps: $f_c(t)=t^2+c$, $c\in K$. Write $\psi_c(t)=\psi_{f_c}(t)$.
	We have:
	\begin{cor}
	 Elements of $\{\psi_c(t):\ c\neq 0,-2\}$ 
	are algebraically independent over $K(t)$.
	\end{cor}
	\begin{proof}
		Let $c\notin\{0,-2\}$, we claim that if $f_c(t)$ is semi-conjugate to $h(t)$
		then $f_c$ and $h$ are conjugate to each other. Let $\pi\in K[t]$ be
		non-constant such that $f_c\circ \pi=\pi\circ h$. Replacing $\pi$ by
		$\pi\circ l$ and $h$ by $l^{-1}\circ h \circ l$, we may assume
		that $h(t)=t^2+\tilde{c}$, for some $\tilde{c}\in K$. We prove by
		contradiction that
		$\tilde{c}=c$, hence establishing the claim.
		
		Assume that $\pi$ has the smallest (positive) degree such that there exists 
		$\tilde{c}\neq c$, and $f_c\circ \pi=\pi\circ h$ for
		$h(t)=t^2+\tilde{c}$. From $\pi(t)^2+c=\pi(t^2+\tilde{c})$, we have that 
		$\pi(-t)=\pm \pi(t)$. If $\pi(-t)=-\pi(t)$ then $\deg(\pi)$ would be odd,
		hence $\gcd(\deg(\pi),\deg(f_c))=1$. By \cite[Theorem 8]{Inou},
		$f_c$ is conjugate to $t^2$ or $t^2-2$. This is impossible when 
		$c\notin\{0,-2\}$. 
		
		Therefore $\pi(-t)=\pi(t)$. Write $\pi(t)=P(t^2)$. We have 
		$(P(t^2))^2+c=P(t^4+2\tilde{c} t^2+\tilde{c}^2)$. Hence $f_c\circ P=P\circ \tilde{h}$,
		where $\tilde{h}(t)=(t+c)^2$. Let $l$ be a linear polynomial such that
		$l^{-1}\circ h\circ l=t^2$. Then $\tilde{\pi}=\pi\circ l$ violates
		the choice of $\pi$, contradiction. This proves the claim at the beginning.
		
		For $c,d\notin\{0,-2\}$
		and $c\neq d$, the polynomials $f_c$ and $f_d$ are not
		conjugate to each other. Therefore, they cannot be semi-conjugate to
		a common polynomial by the claim at the beginning. 
		By Remark \ref{rem:simplecase},
		the functions $\{\psi_c:\ c\neq 0,-2\}$ are algebraically independent
		over $K(t)$. 
		\end{proof}
	
	\section{Some Related Questions}\label{sec:questions}
		\subsection{More than two polynomials.} We expect Theorem 
		\ref{thm:main2} still holds without the extra restriction on the
		degrees. In other words,
		 Theorem \ref{thm:main1} holds for more than two polynomials. However,
		 it does not seem to follow easily from the arguments given in
		 this paper.
		 
		 \subsection{Mahler's method and algebraic independence of values.} For the
		 rest of this paper, we assume $K=\bar{\Q}$ with a fixed 
		 embedding into $\C$.
		 Becker and
		 Bergweiler \cite{BB93} also
		 prove that for every disintegrated $f$, for every $a\in \bar{\Q}\subset \C$
		 such that $|a|$ is sufficiently large,
		  $\psi_f(a)$ is transcendental. Their proof uses ``Mahler's method''
		  in which the transcendence of $\psi_f$ is the first important step
		  so that certain auxiliary polynomials can be constructed
		  (see \cite{Mahler29} or \cite[Chapter 1]{NishiokaBook}).
		 
		 Certain attempts have been made
		 to prove algebraic independence of
		 values of ``Mahler's functions'' $f_1,\ldots,
		 f_n$. As far as we know,
		 only the case of linear functional equations has been treated by various 
		 authors (see \cite[Chapter 3]{NishiokaBook} and the references there). More 
		 precisely, they consider functions $f_1(t),\ldots,f_n(t)$ such that:
		 $$f_i(t^d)=a_{i,1}(t)f_1(t)+\ldots+a_{i,n}f_n(t)+b_i(t)$$
		 for some $a_{i,1},\ldots,a_{i,n},b_i\in \bar{\Q}(t)$, for every $1\leq i\leq 
		 n$.
		 
		 We may ask the following interesting question:
		 \begin{question}\label{q:valuepsi}
		 	Let $f_1,\ldots,f_n$ be disintegrated polynomials of a common degree 
		 	$d\geq 2$. Write $\psi_i=\psi_{f_i}$ for $1\leq i\leq n$. 
		 	Assume that for every $1\leq i\neq j\leq n$, for every $N\geq 1$,
		  the iterates  $f_i^N$ and $f_j^N$ are not semi-conjugate to a common
		  polynomial. Is it true that for every $a\in \bar{\Q}$ such that 
		  $|a|$ is sufficiently large, the numbers $\psi_1(a),\ldots,\psi_n(a)$
		  are algebraically independent over $\Q$?
		 \end{question}
		 
		 Theorem \ref{thm:main2} shows that $\psi_1,\ldots,\psi_n$ are algebraically 
		 independent over $K(t)$. We hope experts in Mahler's method can somehow 
		 utilize our result as a first step in 
		 answering Question \ref{q:valuepsi}.
		 
		\subsection{B\"ottcher coordinates, Green's functions and canonical heights.} 
		The B\"ottcher
		coordinate of $f$ is the 
		inverse
		$\phi_f(t)=\psi_f^{-1}(t)$ which is well-defined up to multiplication
		by a root of unity \cite[Chapter 9]{MilnorBook}. We may ask the following question:
		\begin{question}\label{q:phi}
			Let $f_1,\ldots,f_n$ be disintegrated 
				polynomials of degrees at least 2. Write
				$\phi_i=\phi_{f_i}$ for $1\leq i\leq n$. 
			\begin{itemize}
				\item [(a)] Give a necessary and sufficient
				condition such that $\phi_1,\ldots,\phi_n$
				are algebraically independent over $K(t)$.
				\item [(b)] Suppose $\phi_1,\ldots,\phi_n$ are algebraically
				independent over $\bar{\Q}(t)$. Is it true that
				$\phi(a),\ldots,\phi_n(a)$ are algebraically independent over
				$\bar{\Q}$ for every $a\in\bar{\Q}$ such that 
				$|a|$ is sufficiently large?
			\end{itemize}
		\end{question}
		
		For every polynomial $f$ of degree $d\geq 2$, 
		the dynamical Green's function $G_f$ is defined as:
		$$G_f(a)=\lim_{n\to \infty}\frac{\log|f^n(a)|}{d^n}\ \forall a\in\C.$$
		One could show that $G_f(a)=\log|\phi_f(a)|$ for 
		all sufficiently large $|a|$ \cite[Chapter 9]{MilnorBook}. We conclude this paper by explaining 
		the relation between transcendence of values of Green's functions and
		irrationality of values of canonical heights.
		
		Let $h$ denote the absolute logarithmic height on $\bP^1(\bar{\Q})$
		\cite[Chapter 3]{Sil-ArithDS}. 
		For every rational map $\varphi\in \bar{\Q}(t)$
		of degree $d\geq 2$,
		the canonical height $\hat{h}_\varphi$ is defined as:
		$$\hat{h}_\varphi(a)=\lim_{n\to \infty}\frac{h(\varphi^n(a))}{d^n}\ \forall a\in \bar{\Q}.$$
		It is not difficult to prove that $a$ is $\varphi$-preperiodic if and 
		only if $\hat{h}_\varphi(a)=0$ \cite[Chapter 3]{Sil-ArithDS}. Silverman
		suggests the following question:
		\begin{question}\label{q:transheight}
		Is it true that for every $\varphi$ and for every $\varphi$-wandering $a$,
		the number $\hat{h}_\varphi(a)$ is transcendental?
		\end{question}
		
    
    As far as we know, almost nothing
		is known about Question \ref{q:transheight}. In fact, not much is 
		known
		about the following much weaker:
		\begin{question}\label{q:irrheight}
		Let $f\in \Q[t]$ be a polynomial of degree $d\geq 2$. Is it true that there
		exists $a\in \Q$ such that $\hat{h}_f(a)$ is irrational?
		\end{question}
		
		For $f$ as in Question \ref{q:irrheight}, and for every $a\in \Q$, we have 
		the following:
		\begin{equation}\label{eq:h_f}
		\hat{h}_f(a)=G_f(a)+\log|R_f(a)|,
		\end{equation}
		where $R_f(a)\in \Q$ is the ``contribution from the finite primes'' which
		is relatively easy to handle \cite[Chapter 5]{Sil-ArithDS}. Let $K_f$ denote the filled Julia
		set of $f$ \cite[Chapter 9]{MilnorBook}. If $K_f\cap \R$ has a non-empty interior, one 
		can pick 
		any $f$-wandering $b\in K_f \cap \Q$ so that $G_f(b)=0$, hence 
		$\hat{h}_f(b)=\log|R_f(b)|$
		is irrational (indeed transcendental by the Lindemann-Weierstrass theorem).
		In general when we cannot pick any $f$-wandering $b$ from 
		$K_f\cap \Q$, the difficulty arises from the presence of the archimedean
		contribution $G_f$.
		
		Pick a choice of $\phi_f$. Let $\bar{\phi}_f$ denote the Laurent series 
		obtained from $\phi_f$
		by taking the complex conjugate of each coefficient. Since $f(t)\in \Q[t]$,
		we have $\bar{\phi}_f(f(t))=(\bar{\phi}_f(t))^d$. Hence there exists
		a $(d-1)$th root of unity $\zeta$ such that 
		$\bar{\phi}_f=\zeta\phi_f$. For every $a\in \Q$ such that 
		$|a|$ is sufficiently large, we have:
		\begin{align*}
		G_f(a)=\log|\phi_f(a)|&=\frac{1}{2}\log(\phi_f(a)\bar{\phi}_f(a))
		=\frac{1}{2}\log\zeta(\phi_f(a))^2\\
		&=\frac{1}{2(d-1)}\log(\phi_f(a))^{2(d-1)}.
		\end{align*}
		Therefore, 
		we can rewrite (\ref{eq:h_f}) as:
		\begin{equation}\label{eq:rewriteh_f}
		\hat{h}_f(a)=\frac{1}{2(d-1)}\log\phi_f(a)^{2(d-1)}+\log|R_f(a)|.
		\end{equation}
		
		Let $g\in Q[t]$ be another   
		polynomial. If $a\in \Q$ such that $|a|$ is sufficiently large, and 
		$\displaystyle\frac{\hat{h}_f(a)}{\hat{h}_g(a)}\in \Q$
		then (\ref{eq:rewriteh_f}) implies that there exist positive integers $m$ and $n$ 
		such that 
		$\displaystyle\frac{\phi_f(a)^m}{\phi_g(a)^n}\in \Q$. Therefore, we have
		the following:
		\begin{remark}
			Suppose that part (b) of Question \ref{q:phi} has an affirmative answer.
			Let $f,g\in \Q[t]$ be disintegrated polynomials of degrees at least 2
			such that $\phi_f$ and $\phi_g$ are algebraically independent
			over $\bar{\Q}(t)$. Then for every $a\in \Q$ such that $|a|$ is sufficiently
			large, we have $\displaystyle\frac{\hat{h}_f(a)}{\hat{h}_g(a)}$
			is irrational.
		\end{remark}

	\bibliographystyle{amsalpha}
	\bibliography{Independence} 	

\end{document}